\documentclass[12pt,a4paper,openright,reqno]{extarticle}
\usepackage[english]{babel}
\usepackage{newlfont}
\usepackage{amsbsy}
\usepackage[dvips]{graphicx}
\usepackage{color}
\usepackage{bbm}
\usepackage{graphicx, subfigure}
\usepackage{color}
\usepackage[center]{caption2}
\usepackage{amssymb}
\usepackage{amsmath}
\usepackage{latexsym}
\usepackage{amsthm}
\usepackage{amsthm}
\usepackage[dvips]{geometry}
\theoremstyle{plain}
\newtheorem{thm}{Theorem}

\newtheorem{lem}[thm]{Lemma}
\newtheorem{prop}[thm]{Proposition}
\newtheorem{nota}[thm]{Notation}

\theoremstyle{remark}
\newtheorem{rem}[thm]{Remark}

\newcommand{\Y}{\mathbb{P}}

\newcommand{\Z}{\mathbb{Z}}
\newcommand{\R}{\mathbb{R}}
\newcommand{\N}{\mathbb{N}}

\newcommand{\bp}{\begin{pmatrix}}
\newcommand{\ep}{\end{pmatrix}}

\usepackage{mathtools}
\def\alphaltiset#1#2{\ensuremath{\left(\kern-.2em\left(\genfrac{}{}{0pt}{}{#1}{#2}\right)\kern-.2em\right)}}
\usepackage[center]{caption2}

\begin{document}

\title{Affine subspaces of matrices with  rank in a range}
\author{Elena Rubei}
\date{}
\maketitle

{\footnotesize\em Dipartimento di Matematica e Informatica ``U. Dini'', 
viale Morgagni 67/A,
50134  Firenze, Italia }

{\footnotesize\em
E-mail address: elena.rubei@unifi.it}

\def\thefootnote{}
\footnotetext{ \hspace*{-0.36cm}

}

\begin{abstract} 
The problem of finding the maximal dimension of  linear or affine subspaces of matrices whose rank is constant, or bounded below, or bounded above, has attracted many mathematicians from the sixties to the present day. The problem has caught also the attention of algebraic geometers since  vector spaces of matrices of constant rank $r$ give rise to  vector bundle maps whose images are vector bundles of rank $r$.
Moreover there is a link with the so called ``rank metric codes'', since 
 a constant rank $r$ subspace of $K^{n \times n}$ can be viewed as a constant weight $r$ rank metric code;
  it can be interesting to study also the maximal dimension  of the subspaces of $K^{n \times n}$  whose elements have rank in a range $[s,r]$, since such subspaces obviously  give rank metric codes with weights in $[s,r]$.
In this paper, with the main purpose to get an organic result including the ones on spaces of matrices with constant rank, the ones on spaces of matrices with rank bounded below and the ones on spaces of matrices with rank bounded above and to generalize a previous result on real matrices with constant rank to matrices on a more general field, we study the maximal dimension of affine subspaces of matrices whose rank is between 
two numbers under mild assumptions on the field. We get also a result on antisymmetric matrices and on matrices in row echelon form.
\end{abstract}

\section{Introduction}

For every $m,n \in \mathbb{N}$ and every field $K$, let  $K^{m \times n}$  be the vector space of the $(m \times n)$-matrices over $K$, let $K_s^{n \times n}$ be the vector space of the symmetric $(n \times n)$-matrices over  $K$ and, finally, let $K_a^{n \times n}$ be the vector space of the antisymmetric $(n \times n)$-matrices over  $K$.

We say that an affine subspace $S$ of $K^{m \times n}$ has constant rank $r$ if every matrix of $S$ has rank $r$ and we say that a linear subspace $S$ 
 of $K^{m \times n}$  has constant rank $r$ if every nonzero matrix of $S$ has rank $r$.

The problem of finding the maximal dimension of  linear or affine subspaces of matrices whose rank is constant or bounded below or bounded above has been studied from the sixties to the present day.  It has some interest in algebraic geometry since a $(l+1)$-vector space of $(m \times n )$-matrices of constant rank $r$ gives rise to a vector bundle map ${\cal O}_{\Y^l}^n (-1) \rightarrow {\cal O}_{\Y^l}^m  $ whose image is a vector bundle of rank $r$; furthermore there is a connection with dual varieties; the reader interested in the link between spaces of matrices with constant rank and algebraic geometry can read for instance the paper \cite{I-L}.
Moreover there is a link with the so called ``rank metric codes'', introduced by Delsarte in \cite{Del}, which are subsets of $K^{n \times n}$ together with the weight function given by the metric $d$ defined by  $d(A,B):= rk(A-B)$ for any $A,B \in K^{n \times n}$;
 a constant rank $r$ subspace of $K^{n \times n}$ can be viewed as a constant weight $r$ rank metric code, and it can be interesting to study also the maximal dimension  of a subspace of $K^{n \times n}$  whose elements have rank in a range $[s,r]$, since they obviously  give rank metric codes with weights in $[s,r]$.

To quote some of the main results on the maximal dimension of affine or linear subspaces  of matrices with constant or bounded rank, we need some notation.
Define
$${\cal A}^K(m \times n; r)=  \{ S \;| \; S \; \mbox{\rm  affine subspace of  $K^{m \times n}$ 
 of constant rank }  r\}$$
$${\cal A}_{sym}^K(n;r)=  \{ S \;| \; S \; \mbox{\rm  affine subspace of $K_s^{n \times n}$  of constant rank }  r\}$$
$${\cal A}_{ant}^K(n;r)=  \{ S \;| \; S \; \mbox{\rm  affine subspace of $K_a^{n \times n}$  of constant rank }  r\}$$
$${\cal A}^K(m \times n;s,r)= 
  \left\{ \begin{array}{ll} 
 S \;|  & \; S \; \mbox{\rm  affine subspace of $K^{m \times n}$  such that } \\
 & s =\min  \{rk(A)| \;  A \in S  \}, \;  
r =\max  \{rk(A)| \;  A \in S \}
\end{array}   \right\}$$
 $${\cal A}_{ant}^K(n;s,r)=  \left\{ \begin{array}{ll} 
 S \;|  & \; S \; \mbox{\rm  affine subspace of $K_a^{n \times n}$  such that } \\
 & s =\min  \{rk(A)| \;  A \in S  \}, \;  
r =\max  \{rk(A)| \;  A \in S \}
\end{array}   \right\}
$$
 $${\cal A}_{ech}^K(m \times n;s,r)=  \left\{ \begin{array}{ll} 
 S \;|  & \; S \; \mbox{\rm  affine subspace of $K^{m \times n}$  such that } \\
&   A  \; \mbox{\rm is in row echelon form } \forall A \in S,\\
 & s =\min  \{rk(A)| \;  A \in S  \}, \;  
r =\max  \{rk(A)| \;  A \in S \}
\end{array}   \right\} $$ 
$${\cal L}^K (m \times n;r) =  \{ S \;| \; S \; \mbox{\rm  linear subspace of $K^{m \times n}$   of constant rank }  r\}$$
$${\cal L}_{sym}^K (n; r) =  \{ S \;| \; S \; \mbox{\rm  linear subspace of $K_s^{n \times n}$  of constant rank }  r\},$$
where $rk$ denotes obviously the rank.
Let 
$$a^{\cdot}_{\cdot}(\cdot; \cdot) = \max \{\dim S \mid S \in {\cal A}^{\cdot}_{\cdot}(\cdot; \cdot) ) \}$$
and
$$l^{\cdot}_{\cdot}(\cdot; \cdot) = \max \{\dim S \mid S \in {\cal L}^{\cdot}_{\cdot}(\cdot; \cdot) ) \},$$
for example
$$a_{ant}^K(n;r) = \max \{\dim S \mid S \in {\cal A}_{ant}^K(n,r)  \}.$$

From the wide literature on the maximal dimension of linear subspaces of matrices with constant rank, 
 we quote in particular the following theorems :

\begin{thm}  [Westwick, \cite{W1}] For $2 \leq r \leq m \leq n$, we have:
$$ n-r+1 \leq  l^{\mathbb{C}}(m \times n;r) \leq  m+ n -2 r+1$$
\end{thm}

\begin{thm}  [Ilic-Landsberg, \cite{I-L}] If $r$ is even and greater than or equal to $2$, then
$$l_{sym}^{\mathbb{C}}(n;r) =n-r +1$$
\end{thm}

In case $r$ odd, the following result holds, see \cite{I-L}, \cite{G}, \cite{H-P}:

\begin{thm}  If $r$ is odd, then
$$l_{sym}^{\mathbb{C}}(n;r) =1$$
\end{thm}

We mention also that, in 1962, Flanders
proved the following result:

\begin{thm} [Flanders, \cite{Fl}] \label{Fla}
 If $ r \leq \min\{m, n\}$, a linear subspace of 
 $\mathbb{C}^{m \times n}$
  such that every of its elements has rank less than or equal to $r$ has dimension less than or equal to $r  \max\{m, n\}$.
\end{thm}

The theorem above was generalized by de Seguins Pazzis in 2010:

\begin{thm} [de Seguins Pazzis, \cite{dSP3}] \label{FladSP} Let $K$ be a field. 
 If $ r \leq \min \{m, n\}$, an affine subspace of 
 $K^{m \times n}$ such that every of its elements has rank less than or equal to $r$ has dimension less than or equal to $r  \max \{m, n\}$.
\end{thm}

In \cite{Ru} we proved the following theorems:

\begin{thm} [Rubei] \label{mio}  Let $n,r \in \mathbb{N}$ with $r \leq n$.
Then
$$a_{sym}^{\mathbb{R}}(n;r) \leq
\left\lfloor \frac{r}{2} \right\rfloor
\left(n-  \left\lfloor \frac{r}{2} \right\rfloor\right)
.$$ 
\end{thm}

\begin{thm} [Rubei] \label{mio2}  Let $m,n,r \in \mathbb{N}$ with $r \leq m \leq n$.
Then
$$a^{\mathbb{R}}(m \times n;r) = rn- \frac{r(r+1)}{2} .$$ 
\end{thm}

We proved also a statement 
on  the maximal dimension of affine subspaces with constant signature in $\R_s^{n \times n}$ and one 
on the maximal dimension of affine  subspaces of constant rank in the space on $\R$ of the hermitian matrices.

In \cite{Ru2} we studied the case of antisymmetric matrices and we  proved:

\begin{thm} [Rubei] \label{thmantisym}
For $n \geq 2r+2  $: 
$$a_{ant}^{\mathbb{R}}( n; 2r) = (n-r-1) r   .$$ 
For $n=2r$ 
$$a_{ant}^{\mathbb{R}}( n; 2r) =r(r-1)   .$$ 
For $n=2r+1$ 
$$a_{ant}^{\mathbb{R}}( n; 2r) = r(r+1)   .$$ 
\end{thm}

The result was generalized  
in \cite{dSP2} by de Seguins Pazzis for fields $K$ such that $|K| \geq \max\{ 2r-1, r+2\}$.

Moreover, de Seguins  Pazzis proved the following theorems on subspaces of matrices with rank bounded below 
(see Theorem 8 in \cite{dSP} and Theorem 4 in \cite{dSP2} respectively):

\begin{thm}  [de Seguins  Pazzis, \cite{dSP}] \label{dSP>=} Let $s, m , n \in \N  $ with $s  \leq \min\{m, n\}$ and let $K$ be  a field. The maximal dimension of an affine subspace $S$ in  $K^{m \times n}$  such that $rk(A) \geq s$ for any $A \in S$ is $mn- \binom{s+1}{2}$.

\end{thm}

\begin{thm}  [de Seguins  Pazzis, \cite{dSP2}] \label{dSPant2} Let $s, n \in \N  $ with $s$ even and  $s \leq n $ and let $K$ be  a field such that $|K| \geq n-1$ if $n$ is even and  $|K| \geq  n-2$ if $n$ is odd. The maximal dimension of an affine subspace $S$ in $K^{n \times n}_a  $ such that $rk(A) \geq s$ for any $A \in S$
  is $$\frac{n(n-1)}{2}- \frac{s^2}{4}.$$

\end{thm}

Obviously we can ask which are the maximal dimension of affine or linear subspaces with other characteristics, which are not a specified rank: 
in \cite{Ru3} we proved that the maximal dimension of an affine subspace in the set of the nilpotent $n \times n$ matrices over  a field $K$ 
is $ \frac{n(n-1)}{2}$ and, if the characteristic of the field is zero, 
 an affine not linear  subspace in such a set has dimension  less than or equal to $ \frac{n(n-1)}{2}-1$. 
Moreover we proved that
 the maximal dimension of an affine subspace in the set of the normal $n \times n$ matrices is $n$,
the maximal dimension of a linear subspace in the set of the  $(n \times n)$-matrices over  $\R$ which are diagonalizable over $\R$ is  $ \frac{n(n+1)}{2}$, while the 
maximal dimension of an affine not linear subspace  the set of the  $(n \times n)$-matrices over  $\R$ which are diagonalizable over $\R$ is  $ \frac{n(n+1)}{2} -1$.

In this paper we try to answer to the natural question  ``which is the maximal dimension of affine subspaces of matrices whose rank is bounded both below and above?'' and we prove  the following results.

\begin{thm} \label{mxn} Let $r,s, m , n \in \N  $ with $s \leq r \leq \min\{m, n\}$ and let $K$ be  a field with cardinality  greater than or equal to $r+2$ and characteristic different form $2$;
 then
$$ a^K(m \times n; s,r) = 
r \max \{m,n\} - \binom{s+1}{2}.$$ 
\end{thm}

Observe that the theorem above generalizes Theorem \ref{mio2} in two directions: both because the field we consider is not  only $\R$  and because the rank is not constant but in a range. 
Moreover, observe that, if  we take $r= \min\{m,n\}$ in the  theorem above,  we get Theorem \ref{dSP>=}, while, if we take $s=0$, we get Theorem \ref{FladSP} with more strict assumptions on the field.

\begin{thm} \label{anti} 
Let $r,s, n \in \N  $ with 
 $r,s $ even  and $s \leq r \leq n$ and let $K$ be  a field with cardinality  greater than or equal to $r+2$ and characteristic different form $2$; then
$$a_{ant}^K( n; s,r) \leq   (n-1) \frac{r}{2} - \frac{s^2}{4}
 .$$ 
\end{thm}

Observe that if we take $r=n$ in the theorem above we get  the bound of Theorem
\ref{dSPant2}.

\begin{thm} \label{echel} Let $s,r, m,n \in \N  $ with $ s \leq r \leq \min\{m,n\}$ and let $K$ be  a field with  $|K| \geq r+1$;
 then
$$ \begin{array}{c}a^K_{ech}(m \times n;s, r) = 
r n - \frac{r(r+1)}{2} \;\;\;\;   \mbox{if} \;\;  s=r ,
\vspace*{0.3cm}
\\
sn - \frac{s(s+1)}{2}+ n-s
  \leq a^K_{ech}(m \times n;s, r) \leq 
r n - \frac{r(r+1)}{2}+1 \;\; \;\;   \mbox{if}
 \;\;  s<r.
 \end{array}
$$ 
\end{thm}

Finally we point out that,
as to the techniques we use in this paper,   both because of  the generalization  from $\R$ to a more general field and because of  the generalization from constant rank to rank in a range, we had to use 
 more advanced  techniques with respect the ones used in the previous papers (for instance, in the paper [10] we used that the eigenvalues of real symmetric matrices are real and the nonnegativity of a sum of  squares of real numbers). Moreover  
 we emphasize  that 
 we cannot deduce the results of this paper from the previous ones since the bound we get on the maximal dimension of affine subspaces of matrices with rank in a range is not the minimum of the bound
 on the maximal dimension of affine subspaces of matrices with rank bounded below and 
of the bound
 on the maximal dimension of affine subspaces of matrices with rank bounded  above.

\section{Proof of the theorems}

\begin{nota} \label{notaz} Let $m,n \in \mathbb{N} -\{0\} $ and $K$ be a field.

We denote the $n \times n$ identity matrix over $K $ by $I^K_n$ and 
the $m \times n$ null matrix over $K$ by $0^K_{m \times n }$.
Moreover, we denote $E_{i,j}^{n,K}$ the $n \times n$  matrix over $K$ such that 
$$ (E_{i,j}^{n,K})_{x,y} = \left\{ \begin{array}{ll}
1 & \mbox{\rm if} \; (x,y)=(i,j) \\
0 & \mbox{\rm otherwise}
\end{array} \right.$$
 We omit the superscripts when it is clear from the context.
 
 We write $diag (d_1, \ldots, d_n)$ for the diagonal matrix whose diagonal entries are $d_1, \ldots, d_n$.
 
For any $A \in  K^{m \times n}$, the submatrix of $A$ given by the rows $i_1, , \ldots , i_k$ and the columns $j_1, \ldots, j_s$  will be denoted by $A^{(j_1, \ldots, j_s)}_{(i_1, \ldots, i_k)}$.

We define $J$ to be the $(2 \times 2)$-matrix 
$$ \bp 0 & 1 \\ -1 & 0 \ep $$
and we denote by $\overline{J}_{2n}$ the $(2n \times 2n)$ block diagonal matrix whose diagonal blocks are equal to $J$.
We omit the subscript when it is clear from the context.


Let  $A$ and $B$ be two subsets  of $\N$; we write $A<B$ if $a < b$ for any $a \in A$ and $b \in B$.
\end{nota}

We recall a lemma  from \cite{Ru2}; in \cite{Ru2} it was stated only for $K=\R$ but it can be proved on any field with the same proofs.






\begin{lem} \label{lemma3}
Let $n_1, \ldots, n_k, q_1, \ldots, q_k, m, r \in \N$ and let $K$ be a field.
Let $h=3m + n_1+ \ldots+ n_k$.

Let $\pi_1: K^h \rightarrow K^{2m}$ be the projection onto the first 
$2m$ coordinates.

Let $\pi_2: K^h \rightarrow K^{2m} $ be the projection onto the coordinates $m+1, \ldots, 3m$.

Let $\pi_3 :  K^h \rightarrow K^{2m}$ be the projection 
onto the  coordinates $1, \ldots, m, 2m+1, \ldots 3m$.

Finally, let $p_1 : K^h \rightarrow 
K^{n_1}$ be the projection onto the coordinates $3m+1, \ldots, 3m + n_1$, 
 let $p_2 : K^h \rightarrow 
K^{n_2}$ be the projection onto the coordinates $3m + n_1+1, \ldots, 3m +n_1+n_2$ and so on.

Let $V$ be a vector subspace  of $K^h$ such that $\dim(\pi_i(V)) \leq 2r $ for $i=1,2,3$
and $\dim(p_j (V)) \leq q_j $ for $j=1, \ldots , k$; then
$$\dim(V) \leq \sum_{j=1, \ldots, k} q_j + 3r.$$
\end{lem}



Finally, we recall Schur's Lemma:

\begin{lem} [Schur's Lemma] \label{Schur} Let $K$ be a field and $m,n \in \N-\{0\}$.
Let $A \in K^{n \times n}$, $D \in K^{m \times m}$, $  C \in  K^{m \times n} $, $  B \in  K^{n \times m}$; if $A$ is invertible, then 
$$ \det \bp A & B \\ C & D \ep = \det(A) \det (D - C A^{-1} B);$$ 
if $D$ is invertible, then 
$$ \det \bp A & B \\ C & D \ep = \det(D) \det (A - B D^{-1} C).$$ 

\end{lem}

\begin{proof}[Proof of Theorem   \ref{mxn}]
Let us suppose $ m \leq n$.
To prove the inequality 
$$ a^K(m \times n; s,r) \geq 
r n - \binom{s+1}{2},$$ 
consider the following affine subspace:
$$S=\left\{   
  C \in  K^{m \times n}  \; | \; \begin{array}{l} C_{i,j}=0 \; \mbox{\rm for  } i,j \in \{1, \ldots , s\} \;  \mbox{\rm with } i >j \\ 
  C_{i,i}=1 \; \mbox{\rm for  } i=1, \ldots , s, \;\; 
    C_{i,j}=0 \; \mbox{\rm for  } i >r  
  \end{array}
 \right\}.$$
We can see easily that $s $ is the minimum of $\{rk(C)|\; C \in S\}$ and $r$ is the maximum  and that  the dimension of $S$ is $
r n - \binom{s+1}{2}.$

Now let us prove the other inequality.
Let $S \in {\cal A}^K(m \times n; s,r)$; we prove that $\dim (S) \leq 
r n - \binom{s+1}{2}$.

We can write  $ S=  G+V,$ where $rk(G)=r$ and
 $V$ is a linear subspace of  $K^{m \times n}$.
We can suppose  
$G =\bp 
 I_{r} & 0  \\ 0 & 0 
   \ep   \in  K^{m \times n}$.
Let $$P= \{C   \in   K^{m \times n}  | \; C_{i,j}=0 \; \mbox{\rm if } i > r \; \mbox{\rm and } j  >r \}.$$

First let us prove that 
\begin{equation} \label{w=0}
 V\subseteq P
. \end{equation}
Let   $v \in V$. Let  $i \in \{r+1, \ldots, m \} $ and 
 $j \in \{r+1, \ldots, n \} $. We have that 
$ \det \left( (G + t v )^{(1,\ldots, r, j)}_{(1,\ldots, r, i)} \right)$ is a polynomial in $t$ with coefficient of the term of degree $1$ equal to $v_{i,j} $ and degree at most $r+1$;
since $|K| \geq r+2$ we must have $v_{i,j}=0$ (otherwise 
  there would exist $t$ such that 
$ \det \left( (G + t v )^{(1,\ldots, r, j)}_{(1,\ldots, r, i)} \right)  \neq 0$, which is contrary to our assumption on $S$). 
 Hence we  have proved (\ref{w=0}).

Let $$Q=  \{C   \in P | \;    C_{i,j}=0  \; \mbox{\rm if } \;i \leq r \; \mbox{\rm and }
j  \leq r  \}.$$
Obviously $\dim(Q)= r(m-r)+ r(n-r)$. 
 Let $$\pi : P \longrightarrow Q$$ be the map 
 $$ \bp  A & B & X \\ C & 0  & 0 \ep  \longmapsto  \bp  0 & B & X\\ C & 0 & 0  \ep  ,  
$$ 
where $A $ is $r \times r$ and $B$ is $r \times (m-r)$.
 
 For any $i , j \in \{1, \ldots , m-r\}$,
 let $$\pi_{i,j} : P \longrightarrow K^{2r}$$ be the map 
 $$ \bp  A & B & X \\ C & 0 & 0 \ep  \longmapsto  \bp  B^{(j)} \\ {}^t (C_{(i)})  \ep  ,  
$$ 
where $A $ is $r \times r$ and $B$ is $r \times (m-r)$.

Finally, if $n >m $,  let $$p :  P \longrightarrow K^{r \times (n-m) } $$ be the map 
 $$ \bp  A & B & X \\ C & 0 & 0 \ep  \longmapsto  X,  
$$ 
where, again, $A $ is $r \times r$ and $B$ is $r \times (m-r)$.

Since, for any $ \bp  A & B & X \\ C & 0 & 0 \ep  \in V,  
$ we have that  $\det \bp I_r + t A    & t B^{(j)}  \\ t C_{(i)}  & 0 \ep $ is a polynomial in $t$ with coefficient of the term of degree $2$ equal to $- C_{(i)} B^{(j)}   $ 
(in fact the term of degree $2$  is equal to  $\det \bp I_r    & t B^{(j)}  \\  tC_{(i)}  & 0 \ep $)
and degree at most $r+1$ and since $|K| \geq r+2$, we must have that, for any $i=1, \ldots ,m-r $ and for any $j=  1, \ldots, m-r$, every element of the subspace $\pi_{i,j} (V)= \pi_{i,j} (\pi (V)) $ must be isotropic with respect to the nondegenerate quadratic form on $K^{2r}$ defined by $\sum_{i=1, \ldots, r}
 x_i x_{r+i}$, hence $\dim ( \pi_{i,j} (\pi (V))) \leq r $.
 
 Since
 $$ \pi(V) \subset \pi_{1,1} (\pi(V)) + \ldots +  \pi_{m-r,m-r} (\pi(V)) + p(\pi(V)), $$ 
we get 
$$ \dim (\pi(V)) \leq   \underbrace{r + \ldots +r}_{m-r} + (n-m)r = (m-r)r+ (n-m)r = (n-r)r.$$
Hence there exists a $(m-r) r$-dimensional subspace  $Z$ in $Q$ such that $$\pi(V) \cap Z= \{0\}.$$

Let 
\begin{equation} \label{W}
W =  \{C   \in  K^{m \times n}   | \;    C_{i,j}=0  \; \mbox{\rm if } \;i \geq r+1 \; \mbox{\rm or }
j  \geq r+1  \}
\end{equation}
and let $\tilde{W}$ be a subspace of $W$ such that 
$$ W = (W \cap V) \oplus \tilde{W};$$
hence $V \cap \tilde{W}=\{0\}$.
We state that
\begin{equation}\label{str1} V \cap (Z \oplus \tilde{W})=\{0\}.
\end{equation} 

Let $ \bp  A & 0 & 0 \\ 0 & 0 & 0 \ep  \in \tilde{W} $ and 
 $ \bp  0 & B & X \\ C & 0 & 0 \ep  \in Z$ such that 
$ \bp  A & B & X \\ C & 0 & 0 \ep \in V$, so 
$ \bp  A & B & X \\ C & 0 & 0 \ep \in V \cap    (Z \oplus \tilde{W}) $;
therefore  $$ \bp  0 & B & X \\ C & 0 & 0 \ep = \pi  \bp  A & B & X \\ C & 0 & 0 \ep  \in \pi(V) \cap Z =\{0\},$$ so $B=0$, $X=0$, $C=0$; hence $
 \bp  A & 0 & 0 \\ 0 & 0 & 0 \ep \in V \cap \tilde{W} =\{0\}$, so $A=0$.

From (\ref{w=0}) and (\ref{str1}) we get $$ \begin{array}{l}
\dim(V) \leq \dim(P) - \dim(Z) - \dim(\tilde{W}) = \\ 
= \dim(P) - \dim(Z) - \dim(W) + \dim(W\cap V)= \\
= mn - (m-r)(n-r) -r (m-r) - r^2 +  \dim(W\cap V) 
 \leq \\ \stackrel{\ast}{\leq} 
mn - (m-r)(n-r) -r (m-r) - r^2 +    r^2 - \binom{s+1}{2}  = \\ = 
rn -  \binom{s+1}{2}  ,
\end{array}
$$
where the inequality $(\ast) $ holds by what follows:

let $$U= \left\{ A \in  K^{r \times r } |\;  \bp A & 0 \\ 0 & 0  \ep \in W \cap V\right\}$$
  $$\dim (W \cap V) = \dim(U) = \dim (I_r + U)  \leq  
r^2 - \binom{s+1}{2},$$ where the inequality holds
 by Theorem \ref{dSP>=}.

\end{proof}

\begin{proof}[Proof of Theorem   
\ref{anti}]
Let $S \in {\cal A}_{ant}^K( n; s,r) $; we prove that $\dim (S) \leq  (n-1) \frac{r}{2} - \frac{s^2}{4} $. 

We can write $S$ as
  $M+ V$ where $M \in K_a^{n \times n}$, $rk(M) = r$  and $V$ 
  is a linear subspace of $K_a^{n \times n}$. 
Let  $H$ be an invertible matrix such that ${}^t H M H$ is  a   matrix  $G$  of the kind $$diag(d_1, d_1, d_2, d_2, \ldots, d_{\frac{r}{2}}, d_{\frac{r}{2}}, 1, \ldots, 1) 
\bp 
 \overline{J}_{r} & 0  \\ 0 & 0 
   \ep$$
   with $d_i \in K- \{0\}$ for every $i$. 
   Define $\gamma =diag(d_1, d_1, d_2, d_2, \ldots, d_{\frac{r}{2}}, d_{\frac{r}{2}})  \overline{J}_{r} $. 
 
 Let $$V'= {}^t H V H$$ and $$S'=  {}^t H S H =  G+V'.$$ Obviously $S'  \in   {\cal A}_{ant}^K( n; s,r)  $; moreover, $ \dim(S')= \dim (S)$, so, to prove that $\dim (S) \leq  (n-1) \frac{r}{2} - \frac{s^2}{4} $,
  it is sufficient to prove that $\dim (S') \leq  (n-1) \frac{r}{2} - \frac{s^2}{4} $. We rename $S'$ by $S$ and $V'$ by $V$.

Let $$P= \{C   \in  K_a^{n \times n} | \; C_{i,j}=0 \; \mbox{\rm if } i > r \; \mbox{\rm and } j  >r \}.$$

First let us prove that 
\begin{equation} \label{P}
 V\subseteq P
. \end{equation}
Let   $v \in V$. Let  $i \in \{r+1, \ldots, n \} $ and 
 $j \in \{r+1, \ldots, n \} $. We have that 
$ \det \left( (G + t v )^{(1,\ldots, r, j)}_{(1,\ldots, r, i)} \right)$ is a polynomial in $t$ with coefficient of the term of degree $1$ equal to $
d_1^2 \ldots d_{\frac{r}{2}}^2 
v_{i,j} $ and degree at most $r+1$;
since $|K| \geq r+2$ we must have $v_{i,j}=0$ (otherwise 
  there would exist $t$ such that 
$ \det \left( (G + t v )^{(1,\ldots, r, j)}_{(1,\ldots, r, i)} \right)  \neq 0$, which is contrary to our assumption on $S$).  
 Hence we  have proved (\ref{P}).

Let $$Q=  \{C   \in P | \;    C_{i,j}=0  \; \mbox{\rm if } \;i \leq r \; \mbox{\rm and }
j  \leq r  \}.$$
 
 Let $$\pi : P \longrightarrow Q$$ be the map 
 $$ \bp  A & B \\ -{}^t \! B & 0   \ep  \longmapsto  \bp  0 & B \\ -{}^t B & 0   \ep  ,  
$$ 
where $A $ is $r \times r$.
 
 For any $i, j \in \{1, \ldots , n-r\}$,
 let $$\pi_{i,j} : P \longrightarrow K^{2r}$$ be the map 
 $$ \bp  A & B \\ -{}^t B & 0  \ep  \longmapsto  \bp  B^{(i)} \\ B^{(j)}  \ep  ,  
$$ 
where $A $ is $r \times r$.

Obviously
$\det \bp \gamma + h A    & h B^{(j)}  \\  - h({}^t  B)_{(i)}  & 0 \ep $ must be $0$  for any 
$ \bp  A & B \\ -{}^t B & 0  \ep  \in V $ and for any $h \in K$; moreover
 $\det \bp \gamma+ h A    & h B^{(j)}  \\  - h( {}^t  B)_{(i)}  & 0   \ep $  
 is a polynomial in $h$ of degree less than or equal to $r+1$ and term of degree $2$ equal to 
 \begin{equation} \label{quad} d^2_1 \ldots d^2_{\frac{r}{2}}\sum_{l=1,\ldots, r} \frac{1}{d_l}( 
b_{2l-1,i} b_{2l, j} -b_{2l-1, j} b_{2l ,i}) , 
\end{equation}   (in fact the term of degree $2$  is equal to $\det \bp \gamma    & h B^{(j)}  \\  - h({}^t  B)_{(i)}  & 0 \ep$, which can be  calculated by Schur's Lemma, see Lemma \ref{Schur}). Therefore, since $|K| \geq r+2$, we must have that,  for any $i, j \in \{ 1, \ldots, n-r\}$, every element of the subspace $\pi_{i,j} (V)= \pi_{i,j} (\pi (V)) $ in $K^{2r}$ must be isotropic with respect to the  nondegenerate quadratic form given by (\ref{quad}). Hence 
\begin{equation} \label{dim}
\dim (\pi_{i,j} (\pi (V))) \leq r.
\end{equation}
If $n-r $ is even consider the projections
$\pi_{1,2}, \pi_{3,4}, \ldots, \pi_{n-r-1, n-r}$.
If $n-r $ is odd consider the projections
$\pi_{1,2}, \pi_{1,3}, \pi_{2,3}, \pi_{4,5}, \ldots, \pi_{n-r-1, n-r}$.

By Lemma \ref{lemma3} and from (\ref{dim})
we get that $$\dim(\pi(V)) \leq \frac{r}{2} (n-r).$$
Hence there exists a  $  \frac{r}{2} (n-r)$-dimensional
vector subspace $Z $ in $Q$
such that 
 \begin{equation} \label{Z}
 \pi(V) \cap Z =\{0\}.
 \end{equation}

Let $$W =  \{C   \in K_a^{n \times n} | \;    C_{i,j}=0  \; \mbox{\rm if } \;i \geq r+1 \; \mbox{\rm or }
j  \geq r+1  \}$$
and let $\tilde{W}$ be a subspace of $W$ such that 
$$ W = (W \cap V) \oplus \tilde{W}$$

We state that
\begin{equation}\label{str} V \cap (Z \oplus \tilde{W})=\{0\}.
\end{equation} 
In fact:
let $ \bp  A & 0 \\ 0 & 0  \ep  \in \tilde{W} $ and 
 $ \bp  0 & B \\ -{}^t B & 0 & \ep  \in Z$ such that 
$ \bp  A & B  \\ -{}^t B & 0  \ep \in V$, so 
$ \bp  A & B  \\ -{}^t B & 0  \ep \in V \cap    (Z \oplus \tilde{W}) $;
therefore  $$ \bp  0 & B \\ -{}^t B & 0  \ep = \pi  \bp  A & B  \\ -{}^t B & 0  \ep  \in \pi(V) \cap Z =\{0\},$$ so $B=0$; hence $
 \bp  A & 0  \\ 0 & 0  \ep \in V \cap W \cap \tilde{W}=\{0\}$, so $A=0$.

From (\ref{P}) and (\ref{str}) we get $$ \begin{array}{l}
\dim(V) \leq \dim(P) - \dim(Z) - \dim(\tilde{W}) = \\ = \dim(P) - \dim(Z) - \dim(W) + \dim(W\cap V)
 \\ = 
\frac{n(n-1)}{2} - \frac{(n-r)(n-r-1)}{2} -\frac{r (n-r)}{2} -  \frac{r (r-1)}{2}  + \dim(W\cap V)
 \leq \\ \stackrel{\ast}{\leq} 
\frac{n(n-1)}{2} - \frac{(n-r)(n-r-1)}{2} -\frac{r (n-r)}{2}   -\frac{r (r-1)}{2} +\frac{r(r-1)}{2} -
\frac{s^2}{4}  
 = \\ = \frac{r}{2} (n -1) - \frac{s^2}{4} ,
\end{array}
$$
where the inequality $(\ast) $ holds by what follows:
let $$U= \left\{ A \in  K_a^{r \times r } |\;  \bp A & 0 \\ 0 & 0  \ep \in W \cap V\right\}$$
  $$\dim (W \cap V) = \dim(U) = \dim (\gamma + U)  \leq  
\frac{r(r-1)}{2} -
\frac{s^2}{4}  ,$$ where the inequality holds
 by Theorem \ref{dSPant2}.

\end{proof}

To prove Theorem \ref{echel}, we need the following proposition.

\begin{prop} \label{preechel} Let $m,n \in \N -\{0\}$ and 
 $K$ be a field. 
Let $S$ be an affine subspace in $K^{m \times n} $ such that $A $ is in row echelon form for every $A \in S$. 
Let $$Z(S,i)= \{j \in \{1, \ldots, n\} | \; A_{i,j} \neq 0 \;  \mbox{\rm for some }  A \in S  \}  \;\;\; \forall i=1,\ldots, n, $$
$$P:= \{ i \in \{1, \ldots, m\} |\; Z(S,i) \neq \emptyset     \}$$  and 
$$j_i = \min Z(S,i) \;\;\; \forall i \in P.$$
Let $r =\max  \{rk(A)| \;  A \in S \}$. Then $P=\{1, \ldots, r\}$.
 Moreover, if $|K| \geq r+1 $,
there exists $A \in S$ such that $A_{i, j_i} \neq 0 $ for every $i \in P$. 
\end{prop}

\begin{proof} We can suppose that $S$ contains a nonzero matrix.
Obviously $P$ is a set of the kind $\{ 1,\ldots, l\}$ for some $l \in \{1, \ldots,m\}$.
Since in $S$ there exists a matrix of rank $r$, necessarily $l \geq r$. Moreover,  there  exists a matrix $A$ in $S$ such that $A_{(l)} \neq 0$, but $A$ is in row echelon form, so if $l$ were greater than $r$, we would have $rk(A)\geq l > r$, which is absurd; hence $l=r$. Then $P=\{1, \ldots, r\}$.

We prove by induction on $k$ that there exists $A \in S$ such that 
$A_{i, j_i} \neq 0 $ for $i=1, \ldots , k$ for $k \leq r$. 
For $k=1$ the statement is obvious.

Let us us prove the induction step $k \Longrightarrow k+1$. Let $A \in S$ such that $A_{i, j_i} \neq 0 $ for $i=1, \ldots , k$. 
If $A_{k+1, j_{k+1}} \neq 0 $, there is nothing to prove. 
Suppose $A_{k+1, j_{k+1}} = 0 $; let $A' \in S$ be such that 
$A'_{k+1, j_{k+1}} \neq 0 $. We search for $\lambda \in K$ such that, for $i=1, \ldots, k+1$, $$  (\lambda A + (1- \lambda) A')_{i, j_i} \neq 0 ,$$
 that is 
$$  \lambda ( A_{i, j_i} -  A'_{i, j_i})   \neq  -  A'_{i, j_i} ;$$
 observe that

\begin{itemize}

\item for $i=1, \ldots, k$, if $A_{i, j_i} =  A'_{i, j_i}$, then the inequality is verified for any $\lambda$ because $A'_{i, j_i} =  A_{i, j_i} \neq 0$, while if $A_{i, j_i} \neq A'_{i, j_i}$, then the inequality is obviously equivalent to 
$$  \lambda   \neq  - \frac{  A'_{i, j_i}}{ A_{i, j_i} -  A'_{i, j_i}} ;$$
\item  for $i=k+1$ the inequality is obviously equivalent to $ \lambda \neq 1$.

\end{itemize}

Since $|K| \geq r+1$, we can find $\lambda $ as requested.

The matrix $\lambda A + (1- \lambda) A'$ satisfies the conditions we wanted.
\end{proof}

Now we are ready to prove Theorem \ref{echel}.

\begin{proof}[Proof of Theorem   \ref{echel}]

First consider  the case \underline{$r=s$}.

To prove the inequality 
$ a^K_{ech}(m \times n; r, r) \geq 
r n - \frac{r(r+1)}{2} $,
 consider the affine subspace $$ \left\{ A \in K^{m \times n} |\; 
\begin{array}{l}
A_{i,i}= 1\;  \mbox{\rm for }  i=1,\ldots, r ,\; \\ 
A_{i,j}=0  \;  \mbox{\rm for }  i=r+1,\ldots, m, \; j= 1, \ldots ,n\; \\
A_{i,j}=0  \;  \mbox{\rm for }  i> j\;
\end{array} 
\right\}$$  
It is in $ {\cal A}^K_{ech}(m \times n; r,r)$  and its dimension is $ r n - \frac{r(r+1)}{2}$.

Now let us prove the other inequality. 
Let $S \in {\cal A}^K_{ech}(m \times n; r, r)$.
Let $V$ be the direction of $S$.
Let 
$$P :=\{i \in \{1, \ldots, n\}| \; Z(S,i) \neq \emptyset\}, \hspace*{1cm}
P' :=\{i \in \{1, \ldots, n\}| \; Z(S,i) =\emptyset\},$$
where we use the same  notation as in Proposition 
\ref{preechel}.
By Proposition \ref{preechel} we have that
 $P=\{ 1,\ldots, r\}$ and there exists $H \in S$ such that $H_{i, j_i} \neq 0 $ for every $i  \in P$. 

By simultaneous elementary  column operations on every element of $S$, precisely by operations of the kind ``to add a multiple of a column to a following column'',   we can suppose that the only nonzero entries of $H$ are the pivots. Observe that a matrix in row echelon form remains in row echelon form if we add a multiple of a column to a following column, so the elements of $S$ are still in row echelon form.
  
Let $$L =\{(i,j) \in \{1, \ldots, m\} \times \{1, \ldots, n\}| \; H_{i,j} \neq 0\}=\{(1, j_1), \ldots, (r,j_r)\}.$$
Let
$$W = \left\{ A \in K^{m \times n} |\; 
\begin{array}{ll}
A_{i,j}=0 & \mbox{\rm if }  i  =r+1, \ldots ,m \\
A_{i,j}=0 & \mbox{\rm if } i =1,\ldots, r \; \mbox{\rm and } j < j_i 
\end{array}
\right\}.$$
Obviously  $$\dim (W) \leq rn - \frac{r(r-1)}{2}.$$

Let $$D=\{A \in K^{m \times n} | \; A_{i,j} = 0  \; \forall (i,j) \not\in L\}.$$

We have that $$ V \cap D = \{0\}  ,$$ otherwise 
there would exist in $S$ an element of rank 
less than $r$.

So by Grassmann's formula we get:
$$ 0= \dim(V\cap D) = \dim(V) + \dim(D) -\dim(V+D), $$
hence $$ \begin{array}{l} \dim(V)  
= -\dim(D)+ \dim(V + D)=  \\ = -r+ \dim(V +D)\leq -r +\dim(W) \leq rn - \frac{r(r-1)}{2}- r = rn - \frac{r(r+1)}{2}.
\end{array}$$

Consider now the \underline{case $s <r$.}

To prove the inequality 
$a^K_{ech}(m \times n;s, r) \geq sn - \frac{s(s+1)}{2}+ n-s$, consider the affine subspace $$ \left\{ A \in K^{m \times n} |\; 
\begin{array}{l}
A_{i,i}= 1\;  \mbox{\rm for }  i=1,\ldots, s ,\; \\ 
A_{i,j}=0  \;  \mbox{\rm for }  i=r+1,\ldots, m, \; j= 1, \ldots ,n\; \\
A_{i,j}=0  \;  \mbox{\rm for }  i> j\; \\
A_{s+1, j}= A_{s+1+1, j+1}= \ldots= A_{s+1+r-s, j+r-s}  \;  \mbox{\rm for }   j=s+1, \ldots, n-1\; 
\end{array} 
\right\}$$  
It is in $ {\cal A}^K_{ech}(m \times n; s,r)$  and its dimension is $sn - \frac{s(s+1)}{2}+ n-s$.
  
Now let us  prove the inequality
$a^K_{ech}(m \times n;s, r)    \leq 
r n - \frac{r(r+1)}{2}+1 $. 

Let $S \in {\cal A}^K_{ech}(m \times n; s, r)$, let $V$ be the direction of $S$,  $P$, $P'$  and $W$ be defined as in the case $r=s$.
By Proposition \ref{preechel} we have that
 $P=\{ 1,\ldots, r\}$ and there exists $H \in S$ such that $H_{i, j_i} \neq 0 $ for every $i  \in P$. 
As in the case $r=s$  we can suppose that the only nonzero entries of $H$ are the pivots. 

Define $$U= \langle 
 E_{1,j_1}, \ldots, E_{r-1,j_{r-1}} \rangle.$$
Obviously $\dim(U)=r-1$ and $V \cap U=\{0\}$, in fact: if there existed $ \lambda_1, \ldots , \lambda_{r-1}$ not all zero, for instance such that $ \lambda_{\overline{\i}} \neq 0$, such that 
$$ \lambda_1 E_{1,j_1} + \ldots + \lambda_{r-1} E_{r-1.j_{r-1}} \in V $$ 
then $$ H - \frac{H_{\overline{i}, j_{\overline{i}}}  }{\lambda_{\overline{\i}}} (\lambda_1 E_{1,j_1} + \ldots + \lambda_{r-1} E_{r-1,j_{r-1}} )$$ 
  would be an element of $S$ not in row echelon form.

So by Grassmann's formula we get:
$$ 0= \dim(V\cap U) = \dim(V) + \dim(U) -\dim(V+U), $$
hence $$ \begin{array}{l} \dim(V)  
= -\dim(U)+ \dim(V + U)=  \\ = -r+1+ \dim(V +U)\leq -r+1 +\dim(W) \leq rn - \frac{r(r-1)}{2}- r+1 = rn - \frac{r(r+1)}{2 }+1.
\end{array}$$

\end{proof}

\begin{rem}
Let $r,s, m , n \in \N  $ with $s \leq r \leq m \leq n$ and let $K$ be  a field.
It is natural to wonder what we can say about 
$ a^K(m \times n; s,r)$ if we have not the assumptions 
on $K$ we have in Theorem \ref{mxn}.

It is easy to see that, without the assumptions on the field, the statement of Theorem~\ref{mxn} does not hold any more: consider for instance $K = \Z/3$,
$m=n=3$, $r=s=2$ and the affine subspace
$$ S=\left\{  \bp  a & b & c \\ 0 & a+1 & d \\ 0 & 0 & a+2 \ep | \; a,b,c,d \in \Z/3 \right\} ;$$
every element of $S$ has rank $2$,
the dimension of $S$ is $4$, while $rn - \binom{s+1}{2}= 3$.

Obviously if the cardinality of $K$ is less than $r+2$
and $p(x)$ is a polynomial on $K$ of degree $r+1$, 
 we cannot say that in $K$ there is an element which is not a root of $p$ and it seems difficult to bypass this argument we use in the proof of the theorems of this paper. Anyway, we can easily prove that 
$$ a^K(m \times n; s,r) \leq  
m n  - \binom{s+1}{2},$$ 
and, if the characteristic of $K$ is different from $2$, we can prove also that 
$$ a^K(m \times n; s,r) \leq  
m n  - (m-r)r.$$ The proof is the following.  

\smallskip

Let $S \in {\cal A}^K(m \times n; s,r)$. 
We can suppose   $ S=  G+V,$ where
$G =\bp 
 I_{r} & 0  \\ 0 & 0 
   \ep   \in  K^{m \times n}$ and 
 $V$ is a vector subspace of  $K^{m \times n}$.
Let 
\begin{equation} \label{W}
W =  \{C   \in  K^{m \times n}   | \;    C_{i,j}=0  \; \mbox{\rm if } \;i \geq r+1 \; \mbox{\rm or }
j  \geq r+1  \}
\end{equation}
and let $\tilde{W}$ be a subspace of $W$ such that 
$$ W = (W \cap V) \oplus \tilde{W};$$
hence $V \cap \tilde{W}=\{0\}$. So 
$$ \dim (V) \leq mn - \dim ( \tilde{W})= mn - \dim (W) + \dim (V\cap W) \leq mn - r^2 + r^2 - \binom{s+1}{2},$$ 
where the last inequality holds by Theorem \ref{dSP>=} as in the proof of  Theorem \ref{mxn}.

Moreover let 
\begin{equation} \label{U}
U=  \{C   \in  K^{m \times n}   | \;    C_{i,j}=0  \; \mbox{\rm if } \;i, j \leq r \; \mbox{\rm or }
i,j  \geq r+1  \}
\end{equation}
and let $\tilde{U}$ be a subspace of $W$ such that 
$$ U = (U \cap V) \oplus \tilde{U};$$
hence $V \cap \tilde{U}=\{0\}$. 
 For any $i , j \in \{1, \ldots , m-r\}$,
 let $$\pi_{i,j} : U  \longrightarrow K^{2r}$$ be the map 
 $$ \bp  0_{r \times r} & B  \\ C & 0  \ep  \longmapsto  \bp  B^{(j)} \\ {}^t (C_{(i)})  \ep  ,  
$$ 
and, if $n >m $,  let $$p :  U \longrightarrow K^{r \times (n-m) } $$ be the map 
 $$ \bp  0_{r \times r}  & B \\ C & 0  \ep  \longmapsto  B^{(r+1, \ldots, n)}.
$$ 
 
Since for every  
 $ \bp  0_{r\times r} & B  \\ C & 0  \ep  \in V \cap U
$ and for any $i , j \in \{1, \ldots , m-r\}$, we have that  $\det \bp I_r    &  B^{(j)}  \\  C_{(i)}  & 0 \ep =0$,
 every element of the subspace $\pi_{i,j} (V \cap U)$ must be isotropic with respect to the nondegenerate quadratic form on $K^{2r}$ defined by $\sum_{i=1, \ldots, r}
 x_i x_{r+i}$, hence $\dim ( \pi_{i,j} (V\cap U )) \leq r $.
 
 Since
 $$ V \cap U \subset \pi_{1,1} (V \cap U) + \ldots +  \pi_{m-r,m-r} (V \cap U) + p(V \cap U), $$ 
we get 
$$ \dim (V \cap U) \leq   \underbrace{r + \ldots +r}_{m-r} + (n-m)r = (m-r)r+ (n-m)r = (n-r)r.$$

Hence  
$$ \dim (V) \leq mn  - \dim ( \tilde{U})= mn - \dim (U) + \dim (V\cap U) \leq m n  - (m-r)r. $$

\end{rem}

\vspace*{0.7cm}
{\bf Note.}
In the paper C. De Seguins Pazzis ``On affine spaces of rectangular matrices with constant rank''
 	arXiv:2405.02689  appeared on arXiv in the same period as this paper, the author obtains a result analogous to  Theorem \ref{mxn} in the case $r=s$; the two results have been obtained independently.
 	
\vspace*{0.7cm}

{\bf Acknowledgments.}
This work was supported by the National Group for Algebraic and Geometric Structures, and their  Applications (GNSAGA-INdAM).

The author wishes to thank the anonymous referee for his/her comments, which helped to improve the paper.

\end{document}